\documentclass[12pt]{amsart}

\usepackage{amsmath}
\usepackage{amsrefs}
\usepackage{amssymb}
\usepackage{amsthm}
\usepackage{fancyhdr}
\usepackage{xcolor}
\usepackage[all,cmtip]{xy}
\usepackage{thmtools}
\usepackage{thm-restate}
\usepackage{hyperref}

\theoremstyle{plain}
\theoremstyle{definition}
\theoremstyle{corollary}
\newtheorem{theorem}{Theorem}[section]
\newtheorem*{theorem*}{Theorem}
\newtheorem{lemma}[theorem]{Lemma}
\newtheorem{proposition}[theorem]{Proposition}
\newtheorem{claim}[theorem]{Claim}
\newtheorem{corollary}[theorem]{Corollary}

\newtheorem{fact}[theorem]{Fact}
\newtheorem{maintheorem}{Theorem}

\newtheorem{definition}[theorem]{Definition}
\newtheorem*{definition*}{Definition}
\newtheorem*{lemma*}{Lemma}
\usepackage{graphicx}
\usepackage{pstricks, enumerate, pst-node, pst-text, pst-plot}

\numberwithin{equation}{section}
\numberwithin{theorem}{section}

\newcommand{\R}{\mathbb{R}}
\newcommand{\N}{\mathbb{N}}

\newcommand{\sub}[1]{\mathrm{Sub}_{#1}}

\newcommand{\cT}{{\mathcal{T}}}

\newcommand{\half}{{\textstyle \frac12}}

\newcommand{\sg}[1]{{\left\langle {#1} \right\rangle}}
\newcommand{\nc}[1]{\overline{\sg{#1}}}

\newcommand{\irs}[1]{\mathrm{IRS}(#1)}
\newcommand{\irsg}{\irs{G}}

\newcommand{\modl}[2]{{#1} \backslash {#2}}

\newcommand{\cosg}{\mathrm{Cos}_{G}}
\newcommand{\cosgn}{\mathrm{Cos}_{G,N}}
\newcommand{\cost}{\mathrm{Cos}_{N}}
\newcommand{\closedsubs}[1]{\mathcal{C}({#1})}

\newcommand{\pb}[1]{\Pi\left({#1}\right)}
\newcommand{\ergcom}[2]{{#2} \backslash\kern-.3em\backslash {#1}}
\newcommand{\bow}[2]{\mathrm{B}_{#1}({#2})}

\newcommand{\subg}{\sub{G}}
\newcommand{\subt}{\sub{N}}
\newcommand{\im}{m}

\newcommand{\fk}{H}
\newcommand{\pr}{\mathrm{pr}}
\newcommand{\id}{\mathrm{id}}
\newcommand{\rep}{\mathrm{rep}}
\def\dd{\mathrm{d}}
 
\DeclareMathOperator{\supp}{supp}

\begin{document}

\title[Stabilizer Rigidity in Irreducible Group Actions]{Stabilizer
  Rigidity in Irreducible Group Actions}

\author[Yair Hartman]{Yair Hartman}
\address[Y.\ Hartman]{Weizmann Institute of Science.}

\author[Omer Tamuz]{Omer Tamuz}
\address[O.\ Tamuz]{California Institute of Technology.}

\thanks{Y.\ Hartman is supported by the European Research Council,
  grant 239885.  O.\ Tamuz is supported by ISF grant 1300/08, and is a
  recipient of the Google Europe Fellowship in Social Computing. This
  research is supported in part by this Google Fellowship.}
\date{\today}

\begin{abstract}
  We consider irreducible actions of locally compact product groups,
  and of higher rank semi-simple Lie groups. Using the intermediate
  factor theorems of Bader-Shalom and Nevo-Zimmer, we show that the
  action stabilizers, and all irreducible invariant random subgroups,
  are co-amenable in their normal closure. As a consequence, we derive
  rigidity results on irreducible actions that generalize and
  strengthen the results of Bader-Shalom and Stuck-Zimmer.
\end{abstract}

\maketitle
\tableofcontents

\section{Introduction}

Let $G$ be a locally compact second countable (lcsc) group.  An {\em
  invariant random subgroup} (IRS) of $G$ is a random variable that
takes values in $\subg$, the space of closed subgroups of $G$, and
whose distribution is invariant to conjugation by any element of
$G$~\cite{abert2014kesten}. IRSs arise naturally as stabilizers of
probability measure preserving (pmp) actions, and in fact any IRS is
the stabilizer of some pmp action (see~\cite[Theorem
2.4]{abert2012growth} and also~\cites{abert2014kesten,
  creutz2013stabilizers-of-ergodic}). They are also an interesting
object of study as stochastic generalizations of normal subgroups, and
of lattices.

Let $G = G_1\times G_2$ be a product of two lcsc groups. A pmp action
$G \curvearrowright (X,\im)$ is {\em irreducible} (with respect to the
decomposition $G=G_1 \times G_2)$ if the actions of both $G_1$ and
$G_2$ are ergodic. Likewise, a pmp action of a semi-simple Lie group
is said to be irreducible if the action of every non-central closed
normal subgroup is ergodic. An IRS $K$ in $G$ is irreducible if
$G \curvearrowright (\subg, \lambda)$ is irreducible, where $\lambda$
is the distribution of $K$ and $G$ acts on $\subg$ by
conjugation~\cite{abert2012growth}.

In this paper we study irreducible IRSs of product groups and of
semi-simple Lie groups. Our results are are generalizations of the
theorems of Bader-Shalom~\cite[Theorem 1.6]{bader2006factor} (for
product groups) and of Stuck-Zimmer~\cite{stuck1994stabilizers} (for
semi-simple Lie groups), who both require $G$ to have property (T); we
explore what can be said when this hypothesis is removed.
Nevertheless, in both cases, we rely on the corresponding Intermediate
Factor Theorems: Bader-Shalom~\cite{bader2006factor} and and
Nevo-Zimmer~\cite{nevo2002generalization}.

To state our results we will need to recall the following definition:
A subgroup $H$ is said to be {\em co-amenable} in $G$ if there exists
a $G$-invariant mean on $G/H$~\cite{eymard1972moyennes,
  shalom1999invariant, monod2003co}; equivalently, one can define
co-amenability as a fixed point property or as a representation
theoretical property, in analogy to the different equivalent
definitions of amenability (see~\cite{eymard1972moyennes}
or~\cite[Theorem 4.18]{shalom1999invariant}). A normal subgroup $N
\lhd G$ is co-amenable in $G$ if and only if $G/N$ is amenable.

We say that an IRS $K$ is co-amenable in $G$ if it is almost surely
co-amenable in $G$. Likewise, if $K$ almost surely has some property
(e.g., trivial, normal, co-finite), we say succinctly that $K$ has
this property.

\begin{maintheorem}
  \label{thm:product-irs-rigidity}
  Let $G=G_1 \times G_2$ be a locally compact second countable group,
  and let $K$ be an irreducible IRS in $G$. Then there exist closed
  normal subgroups $N_1 \lhd G_1$ and $N_2 \lhd G_2$ such that $K$ is
  co-amenable in $N_1 \times N_2$.
\end{maintheorem}

\begin{maintheorem}
  \label{thm:lie-irs-rigidity}
  Let $G$ be a connected semi-simple Lie group with finite center, no
  compact factors and $\R$-rank $\ge 2$. Let $K$ be an irreducible IRS
  in $G$. Then $K$ is either equal to a closed normal subgroup, 
  or else $K$ is co-amenable in $G$.
\end{maintheorem}
As these theorems do not require the groups to have property (T), they
provide rigidity results on the irreducible IRSs of groups such as
$SL_2(\R) \times SL_2(\R)$, to which the theorems of Bader-Shalom and
Stuck-Zimmer do not apply. We thus give a partial answer to a question
asked in Stuck-Zimmer~\cite[page 731]{stuck1994stabilizers}. It
remains unknown, however, whether $SL_2(\R) \times SL_2(\R)$ has any
irreducible co-amenable IRSs that are not, in fact, co-finite.

One can interpret Theorem~\ref{thm:product-irs-rigidity} as addressing
the following question: What IRSs does a product group admit? The
irreducibility assumption rules outs the trivial example of a product
of IRSs of each group. Thus, Theorem~\ref{thm:product-irs-rigidity}
says that, in a sense, there is not much else.  A motivation for this
question is the following stronger, basic statement that
holds for normal subgroups:
\begin{fact}
  \label{fact:normal}
  Let $G = G_1 \times G_2$ be a topological group. Given a closed
  normal subgroup $N \lhd G$, let $N_1 \lhd G_1$ and $N_2 \lhd G_2$ be
  the closures of the projections of $N$ to $G_1$ and $G_2$.  Then $N$
  is co-abelian in $N_1 \times N_2$.
\end{fact}

Our approach to these question involves the analysis of the
Furstenberg-Poisson boundary of the random walk on $G$ and on coset
spaces of $G$. Our main technical contribution is in proving the
following claim, which is a generalization to IRSs of an analogous
claim that holds for normal subgroups, but not in general for
non-normal subgroups.
\begin{maintheorem}
  \label{thm:co-amenable-irss-intro}
  Let $G$ be a locally compact second countable group and let $\mu$ be
  a probability measure on $G$ that is equivalent to the Haar
  measure. Let $K \leq G$ be an IRS. If the Furstenberg-Poisson boundary of the
  $\mu$-random walk on $K \backslash G$ is almost surely trivial, then
  $K$ is co-amenable in $G$.
\end{maintheorem}
We in fact prove below a more general result
(Theorem~\ref{thm:co-amenable-irss}) which implies
Theorem~\ref{thm:co-amenable-irss-intro}.

\subsection{Applications}
The conclusions of Theorems~\ref{thm:product-irs-rigidity}
and~\ref{thm:lie-irs-rigidity} can be strengthened when more
constraints are imposed on $G$. In particular, we consider the
following notion: An lcsc group is said to be {\em just non-amenable}
if every closed normal subgroup is co-amenable. Note that if a group
is simple or just non-compact then it is also just non-amenable.

\begin{corollary}
  \label{thm:just-non-amenable}
  Let $G=G_1 \times G_2$ be a locally compact second countable group,
  and let both $G_1$ and $G_2$ be just non-amenable. Then every
  irreducible IRS is either co-amenable in $G$ or equal to a normal
  subgroup.
\end{corollary}
As noted above, this holds in particular when $G_1$ and $G_2$ are
simple. In that case, it is tempting to conjecture (see
Stuck-Zimmer~\cite[page 731]{stuck1994stabilizers}) that in fact every
irreducible IRS is either equal to a normal subgroup or is {\em
  co-finite}; recall that $K \leq G$ is said to be co-finite if there
exists a $G$-invariant finite measure on $G/K$\footnote{In fact,
  co-finite IRSs admit some more structure: any ergodic co-finite IRS
  is supported on a single orbit $\{H^g\}_{g\in G}$, for some
  co-finite $H \le G$ (see Corollary~\ref{cor:tame}). In other words,
  the $G$-action on the IRS is essentially transitive.  }. Bader and
Shalom~\cite{bader2006factor} prove that this is indeed the case when
$G_1$ and $G_2$ both have property (T). Their work continues the work
of Stuck and Zimmer~\cite{stuck1994stabilizers}, who draw the same
conclusions for high rank semi-simple Lie groups whose every simple
factor has property (T).

In the following two corollaries we show that it suffices that only
one of the factors have property (T), both in the product group setting
and in the Lie group setting.
\begin{corollary}
  \label{thm:rigidity-special}
  Let $G = G_1 \times G_2$ be a locally compact second countable
  group, and let $G_1$ be just non-compact and have property (T). Let
  $G \curvearrowright (X,\im)$ be a faithful irreducible pmp action.

  Then the action $G \curvearrowright (X,\im)$ is either essentially
  free or essentially transitive. It follows that the associated
  stabilizer IRS is either trivial or co-finite in $G$.
\end{corollary}
This constitutes a strengthening of the Essentially Free Actions
Theorem of Bader-Shalom; they require that both $G_1$ and $G_2$ have
property (T) and be just non-compact.

\begin{corollary}
  \label{thm:stuck-zimmer}
  Let $G$ be a connected semi-simple Lie group with finite center, no
  compact factors and $\R$-rank $\ge 2$. Assume that one of the simple
  factors of $G$ has property (T). Then any faithful irreducible pmp
  $G$-action is either essentially free or essentially transitive, and
  its associated stabilizer IRS is either trivial or a lattice in $G$.
\end{corollary}

Finally, we show that similar results can be derived without property
(T), given that $G_1$ is simple non-amenable and $G_2$ is simple
discrete.
\begin{corollary}
  \label{thm:discrete}
  Let $G_1$ be a simple, non-amenable, locally compact second
  countable group, and let $G_2$ be a simple, countable, discrete
  group. Then every non-trivial irreducible pmp action of
  $G_1 \times G_2$ is essentially free. It follows that every
  irreducible IRS in $G$ is equal to a normal subgroup.
\end{corollary}

\subsection{Related work}
Recently, Creutz~\cite{creutz2013stabilizers} proved
Corollary~\ref{thm:stuck-zimmer} independently, using a different
approach. In the same paper he also generalizes Bader-Shalom's theorem
to the case that $G_1$ has property (T) and both $G_1$ and $G_2$ are
simple.

Creutz and Peterson~\cite{creutz2013stabilizers-of-ergodic} prove
similar rigidity results for irreducible lattices and commensurators
of lattices in semi-simple Lie groups, and also for product groups
with the Howe-Moore property and property (T).

In~\cite{abert2012growth} it is shown that in the setting of
Corollary~\ref{thm:stuck-zimmer}, if $G$ has property (T) then every
irreducible IRS is either equal to a normal subgroup or is a lattice.


\subsection{Acknowledgments}
We would like to thank Uri Bader, Amos Nevo, Jesse Peterson and
Benjamin Weiss for useful discussions and motivating conversations. We
would also like to thank Yehuda Shalom and Lewis Bowen for helpful
comments on an early draft of this article.

\section{The Chabauty topology and the normal closure of an IRS}
Let $X$ be a locally compact topological space. The space of all
closed subsets of $X$, $\closedsubs{X}$, admits a natural topology
called the {\em compact topology}, under which it is a compact
Hausdorff space (see, e.g.~\cite{wattenberg1977}). If $G$ is a locally
compact group, then $\subg \subset \closedsubs{G}$, the set of closed
subgroups of $G$, is a closed subset and the induced topology on
$\subg$ is known as the Chabauty
topology~\cite{chabauty1950limite}. If $G$ is furthermore second
countable then $\subg$ is a metrizable space, and in particular is
second countable.

In the Chabauty topology, a sequence $\{H_n\}$ of subgroups in $\subg$
converges to $H \in \subg$ if and only if
\begin{enumerate}
\item For every $h\in H$ there exists a sequence $h_n \to h$ such that
  $h_n \in H_n$.
\item If $h_n \to h$ and $h_n \in H_n$ then $h \in H$.
\end{enumerate} 

$G$ acts naturally on $\subg$ by conjugation, and under the Chabauty
topology this action is continuous. An {\em invariant random subgroup}
(IRS) is a Borel probability measure on $\subg$ which is invariant
under this conjugation action. Note that this is a slight (but
standard) abuse of nomenclature; more precisely, an IRS is a random
variable $K$ taking values in $\subg$ whose law is a conjugation
invariant Borel probability measure.  We denote the space of all IRSs
of a given group $G$ by $\irsg$. The term invariant random subgroup
was introduced by Ab{\'e}rt, Glasner, and
Vir{\'a}g~\cite{abert2012growth}, although the mathematical object has
been studied earlier - for example by Stuck and
Zimmer~\cite{stuck1994stabilizers}.

Given an IRS, we want to define its normal closure, which is the
smallest subgroup on which the IRS ``lives''; an IRS $K$ lives in some
$H < G$ if $K$ is almost surely contained in $H$.
\begin{definition}
  Let $K$ be an IRS in $G$ with law $\lambda$.  The {\em normal
    closure} of $\lambda$, denoted $\nc{\lambda}$, is the minimal
  closed subgroup in $G$ that almost surely contains $K$:
  \begin{align*}
    \nc{\lambda} = \min\big\{H \in \subg\,:\,\lambda(\{H'\,:H' \leq
    H\})=1\big\}.
  \end{align*}
\end{definition}
Equivalently,
$\nc{\lambda} = \min\{H \in \subg\,:\,\lambda(\sub{H})=1\}$, where
$\sub{H}$ is the space of closed subgroups of $H$.

It is not obvious from this definition that the normal closure
exists. However, provided that it exists, it is immediate that it is
unique, and applying the conjugation invariance of $\lambda$ yields
that it is normal.

The existence of the normal closure is established in the next
proposition, which also provides an equivalent definition for
it. Before stating the theorem we will introduce the following
notation: For a Borel set $A \subseteq \subg$ we denote by $\sg{A}$
the subgroup of $G$ generated by all the elements of the all groups in
$A$, and by $\nc{A}\in\subg$ the topological closure of this subgroup.
\begin{proposition}
  \label{thm:normal-closure}
  Let $G$ be an lcsc group and let $\lambda \in \irs{G}$. Then
  $\nc{\lambda} = \nc{\supp\lambda}$: the normal closure of $\lambda$
  is equal to the closure of the group generated by all the groups in
  the support of $\lambda$.
\end{proposition}
\begin{proof}
  Let $K$ have distribution $\lambda$, and denote $N = \nc{\lambda}$.
  Since $\subg$ is second countable, $\supp\lambda$ is a
  $\lambda$-full measure set and so $K$ is almost surely a subgroup of
  $N$. It follows that $N \leq \nc{\supp\lambda}$, by the definition
  of the normal closure.

  Thus to prove the claim it suffices to show that that if
  $H \in \subg$ is such that $K$ is almost surely a subgroup of $H$,
  then $\nc{\supp\lambda}$ is a subgroup of $H$.  Note that
  \begin{enumerate}
  \item $\lambda(\sub{H})=1$, since $K$ is almost surely a subgroup of
    $H$.
  \item $\sub{H}$ is a closed set in $\subg$.
  \end{enumerate}
  So $\sub{H}$ is a closed $\lambda$-full measure set, and as
  such must include $\supp \lambda$; here we again use the fact that
  $\subg$ is second countable. It follows that $H$ includes $\nc{\supp
    \lambda}$, and we have proved the claim.

\end{proof}

\section{Poisson bundles and co-amenable IRSs}
\label{sec:bowen-spaces}

\subsection{Random walks on groups and the Furstenberg-Poisson boundary}
Let $G$ be a locally compact second countable group, and let $\mu$ be
a probability measure on $G$ that is equivalent to the Haar measure;
that is, let $\mu$ and the Haar measure be mutually absolutely
continuous\footnote{Some of our intermediate results could be
  generalized to more general $\mu$ (e.g., any $\mu$ absolutely
  continuous with respect to the Haar measure), but we choose not to
  pursue this, in order to simplify the proofs.}. We will tersely say
that $\mu$ is Haar-equivalent.

A $\mu$-random walk on a group $G$ is a measure $\mathbb{P}_\mu$ on
$G^\N$ given by the push-forward of the product measure $\mu^\N$ under
the map $(h_1,h_2,h_3,\ldots) \mapsto (h_1,h_1h_2,h_1h_2h_3,\ldots)$.
Equivalently, let $\{h_n\}_{n \in \N}$ be i.i.d.\ random variables
with measure $\mu$, and let $g_n = h_1 \cdots h_n$. Then a
$\mu$-random walk is the distribution of
$(h_1,h_1h_2,h_1h_2h_3,\ldots)=(g_1,g_2,\ldots)$.

The shift-action on $G^\N$ is given by
$(g_1,g_2,g_3,\ldots) \mapsto (g_2,g_3,\ldots)$. The {\em
  Furstenberg-Poisson boundary}~\cite{furstenberg1971random,
  furstenberg1974boundary} of a $\mu$-random walk, denoted by
$\pb{G,\mu}$, is Mackey's point realization~\cite{mackey1962point} of
the shift-invariant sigma-algebra of $(G^\N,\mathbb{P}_\mu)$ (see,
e.g.,~\cite{zimmer1978amenable, bowen2010random}).

A related process is the $\mu$-random walk on coset spaces of
$G$. Indeed, for any $H \in \subg$, let $\{g_n\}_{n \in \N}$ be as
above. Then $(H g_1, H g_1 g_2,\ldots)$ is a $\mu$-random walk on
$H \backslash G$.

\subsection{Coset spaces}
  
Let $\cosg \subset \closedsubs{G}$ denote the space of all left cosets
of closed subgroups of $G$:
\begin{align*}
  \cosg = \left\{g \fk \,:\, g \in G, \fk \in \subg \right\}.
\end{align*}
As a closed subset in $\closedsubs{G}$, it is naturally equipped with
the corresponding induced topology, and with a continuous $G$ left
action given by $k(g H) = k g H$.

An equivalent definition is to let $\cosg$ be the space of {\em right} cosets
\begin{align*}
  \cosg = \left\{ \fk g \,:\, g \in G, \fk \in \subg \right\}.
\end{align*}
This is indeed equivalent since every left coset $g\fk = g\fk g^{-1}g
= \fk^g g$ is also a right coset. The description of $\cosg$ as a space
of right cosets makes it clear that $G$ also acts on $\cosg$ from the
right by $(Hg)k = H g k$. Note that both the right and left actions on
$\cosg$ are continuous, and that these two actions commute.
  
There are two natural projections $\pi_l,\pi_r\colon \cosg \to \subg$:
\begin{align*}
 \pi_l \colon g H &\mapsto H\\
 \pi_r \colon H g &\mapsto H
\end{align*}
Note that $\pi_r$ is $G$-equivariant with respect to the left
$G$-action on $\cosg$ and $\subg$:
\begin{align*}
  (\pi_r(H g))^k = H^k = \pi_r(H^k k g) = \pi_r(k H g) ,
\end{align*}
and that it is invariant to the right $G$-action:
\begin{align*}
  \pi_r(H g) = H = \pi_r(H g k).
\end{align*}
Similar statements can be made for $\pi_l$.

\subsection{Poisson bundles}
\label{sec:bowenspaces}
Kaimanovich introduces Poisson boundles
in~\cite{kaimanovich2005amenability}. Later, these were also studied
by Bowen in~\cite{bowen2010random}.  Fix $\lambda \in \irsg$, let
$\mu$ be a probability measure on $G$, and consider the space
$\cosg^\N$, endowed with the product topology. $G$ acts on the left by
the diagonal action.

A natural stochastic process on $\cosg^\N$ can be constructed as
follows. Choose $K$ at random from $\lambda$, choose
$(h_1,h_1h_2,\ldots)$ from $\mathbb{P}_\mu$, the $\mu$-random walk on
$G$, and let $C_n \in \cosg$ be given by $K h_1 h_2 \cdots h_n$. Then
$C_n$ has distribution $\lambda * \mu^n$ (here and below $*$ denotes
convolution, and $\mu^n$ are the convolution powers of $\mu$), and it
is easy to verify that $(C_1,C_2,\ldots)$ is a Markov chain on
$\cosg^\N$ with initial distribution $\lambda * \mu$.

To define the measure of this Markov chain formally, equip $\cosg^\N$
with the measure $\mathbb{P}_\mu^\lambda$ given by the push-forward 
$\kappa_*\left(\mathbb{P}_\mu \times \lambda\right)$ where
\begin{equation*}
  \begin{array}{l l c l}
    \kappa \colon &G^\N \times \subg &\longrightarrow& \cosg^\N\\
    &((g_1,g_2,\ldots), H) & \longmapsto & (H g_1, H g_2, \ldots)
\end{array}
\end{equation*}
Then the process $(C_1,C_2,\ldots)$ has distribution $\mathbb{P}^\lambda_\mu$.

It is straightforward to check that $\kappa$ is a $G$-equivariant map
from $G^\N \times \subg$ to $\cosg^\N$, and that
$\mathbb{P}_\mu^\lambda$ is supported on elements of the form
$(\fk g_1, \fk g_2,\ldots)$. On these, the left $G$-action on
$\cosg^\N$ is given by
\begin{align*}
  g(\fk g_1, \fk g_2,\ldots) &=  (g\fk g_1,g\fk
  g_2,\ldots)\\
  &= (\fk^g g g_1,\fk^g g g_2,\ldots).
\end{align*}
Hence the measure $g_*\mathbb{P}_{\mu}^\lambda$ corresponds to
choosing $\fk$ from $\lambda$ and then choosing a $\mu$-random walk on
$\modl{\fk^g}{G}$, starting at $\fk^g g$.  Note, however, that since
$\lambda$ is conjugation invariant, $\fk^g=g\fk g^{-1}$ and $\fk$ have the same
distribution. Note also that since $\mu$ is equivalent to the Haar
measure then $\mathbb{P}_\mu^\lambda$ is $G$-quasi-invariant; in fact,
\begin{align}
  \label{eq:quasi-inv}
  \frac{d g_*(\mathbb{P}_\mu \times \lambda)}{d (\mathbb{P}_\mu \times \lambda)}((g_1,
  g_2,\ldots), H) =
  \frac{d g_*\mu}{d\mu}(g_1),
\end{align}
by the Markov property of the $\mu$-random walk on $G$. Hence
$\mathbb{P}_\mu \times \lambda$ is $G$-quasi-invariant, and so its
push-forward $\mathbb{P}_\mu^\lambda$ is likewise $G$-quasi-invariant.

Note that $\kappa$ is shift-equivariant, if we act in the obvious way
by shifts on $G^\N \times \subg$ and $\cosg^\N$.  The shift-invariant
sigma-algebra defines the Poisson bundle $\bow{\mu}{\lambda}$:
\begin{definition}
  Given $\lambda \in \irsg$, denote by $\bow{\mu}{\lambda}$ Mackey's
  point realization of the shift-invariant sigma-algebra on
  $(\cosg^\N, \mathbb{P}_\mu^\lambda)$.  We shall refer to
  $\bow{\mu}{\lambda}$ as a {\em Bowen space}.
\end{definition}
This presentation is slightly different than
Bowen's~\cite{bowen2010random}, but only semantically so: Bowen
defines $\bow{\mu}{\lambda}$ as a fiber bundle over $(\subg,\lambda)$,
where the fiber over $H < G$ is the Furstenberg-Poisson boundary of the
$\mu$-random walk on $H \backslash G$. This is simply the disintegration of
$\bow{\mu}{\lambda}$ with respect to the factor
$\hat\pi_r \colon \bow{\mu}{\lambda} \to (\subg,\lambda)$ defined in
Proposition~\ref{prop:bowen-space-if}; see
also~\eqref{eq:bowen-space-decompo} and the preceding paragraph. We
encourage the reader to study the details in Bowen’s
paper~\cite{bowen2010random}. For further discussion and another
application of Poisson bundles see~\cite{hartman2012furstenberg}.

Another point of view is that, as the Mackey realization of the
shift-invariant sigma-algebra, the Poisson bundle $\bow{\mu}{\lambda}$ is
the Furstenberg-Poisson boundary of the Markov chain $(C_1,C_2,\ldots)$ described
in the beginning of this section; this is simply the definition of the
Furstenberg-Poisson boundary of a Markov chain. Since $G$ acts on $\cosg^\N$,
Mackey's realization provides us also with a $G$-action on
$\bow{\mu}{\lambda}$. This action can be interpreted as follows: The
application of $g \in G$ to $\bow{\mu}{\lambda}$, the Furstenberg-Poisson boundary
of the Markov chain with initial distribution $\lambda * \mu$, yields
the Furstenberg-Poisson boundary of the same Markov chain, with initial
distribution $g_*(\lambda * \mu) = \lambda * (g_*\mu)$.

\subsection{Co-amenable IRSs}
\label{sec:co-amenable}
The following result is due to
Kaimanovich~\cite{kaimanovich2002thepoisson}.
\begin{theorem}
  \label{clm:normal-co-amenable}
  Let $G$ be an lcsc group, and let $\mu$ be a Haar-Equivalent
  probability measure on $G$.  Let $N' \leq N$ be two closed normal
  subgroups of $G$, and let $\bar\mu$ be the projection of $\mu$ to
  $G/N'$. If the $N$-action on the Furstenberg-Poisson boundary
  $\pb{G/N',\bar\mu}$ is measure preserving, then $N'$ is co-amenable
  in $N$.

  In particular, if the Furstenberg-Poisson boundary of the $\bar\mu$-random walk
  on $G/N'$ is trivial, then $N'$ is co-amenable in $G$.
\end{theorem}

This theorem does not hold in general for non-normal subgroups. In
this section we prove the following theorem, which shows that it does
hold for IRSs.
\begin{theorem}[Co-amenable IRSs]
  \label{thm:co-amenable-irss}
  Let $G$ be an lcsc group, and let $\mu$ be a Haar-Equivalent
  probability measure on $G$. Let $K \leq G$ be an IRS with
  distribution $\lambda$, and let $N \lhd G$ be a closed normal group
  such that almost surely $K \leq N$. If the $N$-action on the Poisson
  bundle $\bow{\mu}{\lambda}$ is measure preserving, then $K$ is
  co-amenable in $N$.

  In particular, if the Furstenberg-Poisson boundary of the $\mu$-random walk on
  $K \backslash G$ is almost surely trivial, then $K$ is co-amenable
  in $G$.
\end{theorem}
The rest of this section is devoted to proving this theorem.

The next lemma can be deduced from known general results about the
Furstenberg-Poisson boundaries of Markov chains (see, e.g.,~\cite[Lemma
2.1]{kaimanovich1992measure}). We provide its proof here for the
reader's convenience. Recall that we assume throughout that $\mu$ is
equivalent to the Haar measure. We denote by $\|\cdot\|$ the total
variation norm.
\begin{lemma}
  \label{lem:invariant-finitary}
  Let $(B,\nu) = \bow{\mu}{\lambda}$. If $g_*\nu = \nu$ for some $g
  \in G$ then
  \begin{align*}
    \lim_n \|g_*\lambda * \mu^n - \lambda * \mu^n\| = 0.
  \end{align*}
\end{lemma}
Recall that $\lambda * \mu^n$ is the projection of
$\mathbb{P}_\mu^\lambda$ on the $n$\textsuperscript{th} coordinate, or
the position of the random walk at time $n$. Likewise, $g_*\lambda *
\mu^n = \lambda * g_* \mu^n$ is the position at time $n$ when the
initial distribution of the random walk is $\lambda * g_*\mu$ rather
than $\lambda * \mu$.

Intuitively, the Furstenberg-Poisson boundary distribution is the distribution of
the random walk at time infinity. The theorem hypothesis $g_*\nu =
\nu$ means that the Furstenberg-Poisson boundary is unchanged when the random walk
is initially displaced by $g$. The claim is that under the same
displacement, the distributions of the positions at large times $n$
are also similar.

\begin{proof}
  Let $\cT_n^m$ be the sigma-algebra of
  $(\cosg^\N,\mathbb{P}_\mu^\lambda)$ consisting of the events
  measurable in coordinates $(n,n+1,\ldots,m)$, and let $\cT =
  \cap_n\cT_n^\infty$ be the tail sigma-algebra. We first note that
  $\cT$ and the shift-invariant sigma-algebra coincide, mod
  $\mathbb{P}_\mu^\lambda$ null sets. This is not true for general
  Markov chains, but it does hold for random walks on groups (see,
  e.g.,~\cite{kaimanovich1983random}), from which it easily follows
  that it also holds here.

  Hence $\bow{\mu}{\lambda}$ can be taken to be the Mackey point
  realization of $\cT$, the tail sigma-algebra on
  $(\cosg^\N, \mathbb{P}_\mu^\lambda)$.  Since
  $g_*\mathbb{P}_\mu^\lambda$ and $\mathbb{P}_\mu^\lambda$ are
  equivalent (see~\eqref{eq:quasi-inv} and the subsequent paragraph),
  the hypothesis $g_*\nu = \nu$ means that
  \begin{align}
    \label{eq:tail-events}
    \mathbb{P}_\mu^\lambda(g^{-1} T) = \mathbb{P}_\mu^\lambda(T)
    \mbox{ for every tail event } T \in \cT.
  \end{align}

  Let
  \begin{align*}
    d_n = \sup_{T \in \cT_n^\infty}\|g_*\mathbb{P}_\mu^\lambda(T) -
    \mathbb{P}_\mu^\lambda(T)\|
  \end{align*}
  be the total variation distance between
  $g_*\mathbb{P}_\mu^\lambda(T)$ and $\mathbb{P}_\mu^\lambda(T)$,
  taken as measures over $(\cosg^\N,\cT_n^\infty)$. Then by (say) the
  Martingale convergence theorem,
  \begin{align*}
    \lim_n d_n = \sup_{T \in \cT}\|g_*\mathbb{P}_\mu^\lambda(T) -
    \mathbb{P}_\mu^\lambda(T)\| = 0,
  \end{align*}
  where the second equality follows from~\eqref{eq:tail-events}.

  By the Markov property of this $\mu$-random walk, one can take the
  supremum in the definition of $d_n$ to be only over events
  measurable in the $n$\textsuperscript{th} coordinate only:
  \begin{align*}
    d_n = \sup_{T \in \cT_n^n}\|g_*\mathbb{P}_\mu^\lambda(T) -
    \mathbb{P}_\mu^\lambda(T)\|.
  \end{align*}
  This expression is in turn equal to
  \begin{align*}
    \half\|g_*\lambda * \mu^n - \lambda * \mu^n\|,
  \end{align*}
  since the projection of $\mathbb{P}_\mu^\lambda$ on its
  $n$\textsuperscript{th} coordinate is $\lambda * \mu^n$.  Thus
  \begin{align*}
    \lim_n \|g_*\lambda * \mu^n - \lambda * \mu^n\| = 2\lim_n d_n = 0.
  \end{align*}
\end{proof}

A sequence of probability measures $\{\zeta_n\}$ on a $G$-space is
{\em almost-invariant} if
\begin{align*}
  \lim_n\|g_*\zeta_n - \zeta_n\| = 0
\end{align*}
for all $g \in G$.

\begin{corollary}
  \label{cor:t-inv-cosg}
  In the setting of Theorem~\ref{thm:co-amenable-irss}, there exist
  $N$-almost invariant probability measures $\{\zeta_n\}$ on $\cosg$
  such that $\pi_{l*} (g_* \zeta_n) = \pi_{r*} (g_* \zeta_n) =
  \lambda$ for all $n \in \N$ and $g \in G$.
\end{corollary}
\begin{proof}
  Let $\zeta_n = \lambda * \mu^n$, and let
  $(B,\nu) = \bow{\mu}{\lambda}$.  Since $\nu$ is invariant to the
  $N$-action, it follows from Lemma~\ref{lem:invariant-finitary} that
  for every $h \in N$
  \begin{align*}
    \lim_n\|h_*\zeta_n - \zeta_n\| = 0,
  \end{align*}
  or, in other words, that the sequence $\{\zeta_n\}$ is $N$-almost
  invariant.

  Recall that $\pi_r(H g) = H$ and $\pi_l(g H) = H$ for every $g H, H
  g \in \cosg$. Since
  \begin{align*}
   \pi_{r*}(g_*\zeta_n) = \pi_{r*}(\lambda * (g_*\mu^n)) = \lambda, 
  \end{align*}
  we have that $\pi_{r*}\zeta_n = \lambda$. And since $\lambda$ is
  conjugation invariant, we have that $\lambda * (g_*\mu^n) =
  (g_*\mu^n) * \lambda$, from which it follows that $\pi_{l*}\zeta_n =
  \lambda$.
\end{proof}

In order to prove that $\lambda$ is co-amenable in $N$, we will
demonstrate the existence of $N$-almost invariant measures on $N / H$,
for $\lambda$-almost every $H$. These will be constructed as
disintegrations of the $N$-almost invariant measures $\{\zeta_n\}$ on
$\cosg$ from Corollary~\ref{cor:t-inv-cosg}. Our first task will be to
push the measures $\{\zeta_n\}$ forward to $N$-almost invariant
measures on $\cost$.

To this end, let $\rep \colon G/N \to G$ be a measurable map which
satisfies $\rep(g N) \in g N$; this is a section of the quotient map
$g \mapsto g N$, or a map that chooses a representative for each coset
of $N$ in $G$ in a measurable fashion. The existence of such a map was
shown by Mackey~\cite[Lemma 1.1]{mackey1952induced}. Let
$\rho \colon G \to N$ be given by
\begin{align*}
  \rho(g) = g \cdot \rep\left(g^{-1} N\right).
\end{align*}
Intuitively, $\rho(g)$ is the ``difference'' between $g$ and the
representative of its $N$-coset.  Note that $\rho$ is measurable and
equivariant to the left $N$-action: for all $g \in G$ and $k \in N$ it
holds that $\rho(k g) = k\rho(g)$.

Let $\cosgn$ be the restriction of $\cosg$ to $G$-cosets of subgroups
of $N$:
\begin{align*}
  \cosgn = \{H g \in \cosg \,:\,H \in \subt\}.
\end{align*}
This closed subspace is equipped with a left $N$-action: for $k \in N$
and $H g \in \cosgn$, $k(H g) = H^k k g \in \cosgn$. This is simply the
restriction of the $G$-action on $\cosg$.

For $H g \in \cosgn$ it follows from the $N$-equivariance of $\rho$
that $\rho(H g) = H \rho(g) \in \cost$. Hence we can extend $\rho$ to
a map $\rho \colon \cosgn \to \cost$. It is easy to check that this
map too is equivariant with respect to the left $N$-action.

Now, $\pi_{r*} (g_* \zeta_n) = \lambda$ by
Corollary~\ref{cor:t-inv-cosg}, and so our measures $\zeta_n$ are
supported on $\cosgn$. We can thus use $\rho$ to push-forward our
$N$-almost invariant measures $\zeta_n$ on $\cosgn$ to $N$-almost
invariant measures on $\cost$: $\alpha_n = \rho_*\zeta_n$. By again
invoking the $N$-equivariance of $\rho$, it also follows that the
projections $\pi_{r*} \alpha_n$ and $\pi_{l*} \alpha_n$ are equal to
$\lambda$ (now as a measure on $\subt$). We have thus proved the
following claim.
\begin{claim}
  \label{lem:t-invariant-on-tlam}
  There exist $N$-almost invariant probability measures $\alpha_n$ on
  $\cost$ such that $\pi_{l*} (h_* \alpha_n) = \pi_{r*} (h_* \alpha_n) =
  \lambda$ for all $n \in \N$ and $h \in N$.
\end{claim}

We are now ready to take the last step in the proof of
Theorem~\ref{thm:co-amenable-ift}.
\begin{claim}
  \label{clm:co-amenable-in-t}
  $\lambda$ is co-amenable in $N$.
\end{claim}
\begin{proof}
  Let $\alpha_n$ be a sequence of probability measures on $\cost$
  given by Claim~\ref{lem:t-invariant-on-tlam}.  Recall that $\pi_l
  \colon \cost \to \subt$ is given by $\pi_l(t H) = H$, and that
  $\pi_{l*}(t_*\alpha_n) = \lambda$ for all $t \in N$. We fix $t \in
  N$, and disintegrate both $\alpha_n$ and $t_*\alpha_n$ with respect
  to $\pi_l$:
  \begin{align*}
    \alpha_n &= \int_{\subt}\alpha_n^H \dd\lambda(H)\\
    t_*\alpha_n &= \int_{\subt}(t_*\alpha_n)^H \dd\lambda(H),
  \end{align*}
  so that $\alpha_n^H$ and $(t_*\alpha_n)^H$ are measures on $N/H$.
  Note that for every $k,k' \in N$ it holds that
  $\pi_l(k k' H) = \pi_l(k' H) = H$.  Hence both $\alpha_n^H$ and
  $t_* (\alpha_n^H)$ are supported on the same fiber (namely $N / H$),
  and we get that $(t_*\alpha_n)^H = t_*(\alpha_n^H)$.

  As both $t_*\alpha_n$ and $\alpha_n$ are projected by $\pi_l$ to
  $\lambda$, we can disintegrate $t_*\alpha_n -\alpha_n$ to get
  \begin{align*}
    t_*\alpha_n-\alpha_n &= \int_{\subt}\left((t_*\alpha_n)^H-\alpha_n^H\right)\dd\lambda(H) \\
    &= \int_{\subt} \left(t_*(\alpha_n^H)-\alpha_n^H\right) \dd\lambda(H),
  \end{align*}
  and in particular, 
  \begin{align*}
    \|t_*\alpha_n-\alpha_n\| =
    \int_{\subg}\left\|t_*(\alpha_n^H)-\alpha_n^H\right\|\dd\lambda(H).
  \end{align*}
  So for any $t\in N$ we have that $\|t_*\alpha_n-\alpha_n\| \to 0$,
  and therefore $\left\|t_*(\alpha_n^H)-\alpha_n^H\right\| \to 0$ for
  $\lambda$-almost every $H$. Finally, the existence of asymptotically
  invariant measures on $N/H$ implies that $H$ is co-amenable in $N$
  (see~\cite{eymard1972moyennes} or~\cite[Theorem
  4.18]{shalom1999invariant}).
\end{proof}

This completes the proof of Theorem~\ref{thm:co-amenable-irss}, except
for the last statement. To see that it holds, assume that the Furstenberg-Poisson
boundary of the $\mu$-random walk on $G/K$ is almost surely
trivial. We now invoke Bowen's definition of the Poisson bundle
$\bow{\mu}{\lambda}$ as a bundle over $(\subg,\lambda)$, where the
fiber over $H < G$ is the Furstenberg-Poisson boundary of the $\mu$-random walk on
$H \backslash G$ (see~\cite{bowen2010random}, the discussion in the
penultimate paragraph of Section~\ref{sec:bowenspaces}, as well
as~\eqref{eq:bowen-space-decompo} and the preceding discussion). Since
these fibers are all trivial, it follows that $\bow{\mu}{\lambda}$ is
isomorphic, as a $G$-space, to $(\subg,\lambda)$. In particular the
$G$-action on the $\bow{\mu}{\lambda}$ is measure preserving. Hence,
by the argument above, $K$ is co-amenable in $G$.

\section{Intermediate factor theorems}
Our main results, Theorems~\ref{thm:product-irs-rigidity}
and~\ref{thm:lie-irs-rigidity}, are consequences of
Theorem~\ref{thm:co-amenable-ift} below, which is a statement
regarding any IRS that satisfies an {\em intermediate factor theorem}
(IFT).  The original proofs of Stuck-Zimmer and of Bader-Shalom are
also each based on a corresponding IFT. We start this section by
defining intermediate factors.
 
Let $\pb{G,\mu}$ be the Furstenberg-Poisson boundary of a group $G$ with a
Haar-equivalent probability measure $\mu$, and let
$G \curvearrowright (X,\im)$ be a pmp action. A $G$-quasi-invariant
probability space $(Y,\eta)$ is a $(G,\mu)$-{\em intermediate factor
  over $(X,\im)$} if there exist $G$-factors $\kappa$ and $\pi$
\begin{align*}
\xymatrix{  \pb{G,\mu} \times (X,\im) \ar[d]^{\kappa}\\  (Y,\eta)
  \ar[d]^{\pi} \\ (X,\im)}
\end{align*}
such that the composition $\pi \circ \kappa$ is the natural projection
$(p,x) \mapsto x$. Note that, if we denote by $\nu$ the measure of the
Furstenberg-Poisson boundary $\pb{G,\mu}$, then $\kappa_*(\nu \times \im) =
\eta$.
$\pi_*\eta = \im$ and $[\pi \circ \kappa]_*(\nu \times \im) = \im$.

A trivial example of an intermediate factor is
\begin{align*}
  \xymatrix{\pb{G,\mu} \times (X,\im) \ar[d]^{\kappa_0\times
    \id}\\ (Z,\xi) \times (X,\im)
  \ar[d]^{\pr}\\ (X,\im)}
\end{align*}
where $(Z,\xi)$ is any $G$-factor of the Furstenberg-Poisson boundary. $G$-factors
of the Furstenberg-Poisson boundary $\pb{G,\mu}$ are also called $(G,\mu)$-{\em
  boundaries}, or $\mu$-proximal actions.

Zimmer~\cite{zimmer1982ergodic} proves an {\em intermediate factor
  theorem}, which was generalized (and had its proof corrected) by
Nevo and Zimmer~\cite{nevo2002generalization}.  It provides conditions
on $G$ and $(X,\im)$ under which every intermediate factor is
isomorphic to a product $(Z,\xi) \times (X,\im)$.  Bader and
Shalom~\cite{bader2006factor} prove the same result for intermediate
factors over irreducible actions of product groups. To state these
theorems we first define an {\em IFT action}.
\begin{definition}
  A pmp $G$-action $G \curvearrowright (X,\im)$ is an {\em IFT action}
  if there exists a Haar-equivalent $\mu$ such that for every
  $(G,\mu)$-intermediate factor
  \begin{align*}
    \xymatrix{  \pb{G,\mu} \times (X,\im) \ar[d]^{\kappa}\\  (Y,\eta)
      \ar[d]^{\pi} \\ (X,\im)}
  \end{align*}
  there exists a $G$-factor $(Z,\xi)$ of the Furstenberg-Poisson boundary and
  $G$-isomorphism $\varphi \colon (Y,\eta) \to (Z,\xi)\times(X,\im)$ such
  that the following diagram commutes.
  \begin{align*}
    \xymatrix{  \pb{G,\mu} \times (X,\im)
      \ar[d]^{\kappa}
      \ar[dr]^{\kappa_0 \times \id}&\\
      (Y,\eta)
      \ar[d]^{\pi} \ar[r]^{\varphi}& (Z,\xi)\times(X,\im)
      \ar[dl]^{\pr} \\
      (X,\im)&}
  \end{align*}  
\end{definition}

\begin{theorem}[Intermediate Factor Theorem~\cite{zimmer1982ergodic,
    bader2006factor}]
  \label{thm:ift}
  Let $G$ be either (1) a product of lcsc groups, or (2) a connected
  semi-simple Lie group with finite center, no compact factors and
  higher rank. Then every irreducible pmp $G$-space is an IFT
  action.
\end{theorem}
Recently, Levit~\cite{arie2014nevo} proved an intermediate factor
theorem over local fields.

\subsection{Poisson bundles as intermediate factors}
The next claim shows that every Poisson bundle $\bow{\mu}{\lambda}$ is an
intermediate factor. This will allow us to apply intermediate factor
theorems to Poisson bundles.
\begin{proposition}
  \label{prop:bowen-space-if}
  Let $\lambda \in \irsg$. Then $\bow{\mu}{\lambda}$ is an
  intermediate factor. Namely, there exist $G$-maps
  \begin{align}
    \label{eq:bowen-intermediate2}
    \xymatrix { \pb{G,\mu} \times (\subg,\lambda)
      \ar[d]^{\hat\kappa}\\  \bow{\mu}{\lambda}
      \ar[d]^{\hat\pi_r} \\ (\subg, \lambda)}
  \end{align}
  such that $\hat\pi_r \circ \hat\kappa$ is the projection
  $(b, H) \mapsto H$.
\end{proposition}

\begin{proof}
  Recall that $\pi_r \colon \cosg \to \subg$ is given by
  $\pi_r(Hg)=H$, and that it commutes with the left $G$-action:
  $\pi_r(k H g) = \pi_r(Hg)^k$. We extend its definition to a $G$-map
  $\pi_r \colon \cosg^\N \to \subg$ by
  \begin{align*}
    \pi_r(c_1,c_2,\ldots) = \pi_r(c_1).
  \end{align*}
  Recall that $\kappa \colon G^\N \times \subg \to \cosg^\N$ is given
  by $\kappa((g_1,g_2,\ldots), H) = (H g_1, H g_2, \ldots)$.  Hence
  the composition $\pi_r \circ \kappa$
  (see~\eqref{eq:bowen-intermediate} below) is equal to the projection
  $((g_1,g_2,\ldots),H) \mapsto H$.
  \begin{align}
    \label{eq:bowen-intermediate}
    \xymatrix { (G^\N \times \subg,  \mathbb{P}_\mu \times \lambda)
      \ar[d]^{\kappa}\\  (\cosg^\N,\mathbb{P}^\lambda_\mu)
      \ar[d]^{\pi_r} \\ (\subg, \lambda)}
  \end{align}
  
  Define the shift-action on $G^\N \times \subg$ in the obvious way.
  Then the shift commutes with $\kappa$, and so the shift-invariant
  sigma-algebra of $(\cosg^\N,\mathbb{P}^\lambda_\mu)$ is a
  sub-sigma-algebra of the shift-invariant sigma-algebra of
  $(G^\N \times \subg, \mathbb{P}_\mu \times \lambda)$. Hence
  $\kappa$, as a factor between the two shift-invariant
  sigma-algebras, extends to a factor $\hat\kappa$ between the Mackey
  realizations of these two sigma-algebras, which are, by definition,
  $\pb{G,\mu} \times (\subg,\lambda)$ and $\bow{\mu}{\lambda}$.
  
  As we note above, $\mathbb{P}^\lambda_\mu$ is supported on elements
  of the form $(H g_1, H g_2,\ldots)$. Since
  $\pi_r(H g_1, H g_2,\ldots) = H$, $\pi_r$ is in fact measurable in
  the shift-invariant sigma-algebra of
  $(\cosg^\N, \mathbb{P}_\mu^\lambda)$. It follows that, as with
  $\kappa$ above, $\pi_r$ can be extended to a $G$-map $\hat\pi_r$ of
  the corresponding Mackey realizations; namely to a $G$-map from
  $\bow{\mu}{\lambda}$ to $(\subg,\lambda)$. Thus $\hat\kappa$ and
  $\hat\pi_r$ act as in~\eqref{eq:bowen-intermediate2}.
  
  Finally, since $\pi_r \circ \kappa$ is the projection on the
  second coordinate, then so is $\hat\pi_r \circ \hat\kappa$, and we
  have proved the claim.
\end{proof}

Having shown that every Poisson bundle is an intermediate factor, we
explore the consequences of the application of an intermediate factor
theorem to one. In particular, when
$G \curvearrowright (\subg, \lambda)$ is an IFT action,
$\bow{\mu}{\lambda}$ is isomorphic to the product
$(\subg,\lambda) \times (Z,\xi)$, where $(Z,\xi)$ is some
$(G,\mu)$-boundary. As the next proposition shows, this means that
$(Z,\xi)$ is invariant to every element of the normal closure of
$\lambda$.
\begin{proposition}
  \label{prop:bowen-ift}
  Let $\lambda \in \irsg$ be such that $G \curvearrowright (\subg,
  \lambda)$ is an IFT action. Then $\bow{\mu}{\lambda} = (Z,\xi)
  \times (\subg,\lambda)$, where $(Z,\xi)$ is $\nc{\lambda}$-invariant.
\end{proposition}

Before proving this proposition we will need to introduce some
additional notation. A natural decomposition of the Poisson bundle
$\bow{\mu}{\lambda}$ is by disintegration according to the factor
$\hat\pi_r \colon \bow{\mu}{\lambda} \to (\subg,\lambda)$ defined in
the proof of Proposition~\ref{prop:bowen-space-if} above.  We denote
the fiber $\hat\pi_r^{-1}(H)$ by $(B_H,\nu_H)$; this space can be shown to
be the Furstenberg-Poisson boundary of the $\mu$-random walk on
$H \backslash G$~\cite{bowen2010random}. Hence we can write
\begin{align}
  \label{eq:bowen-space-decompo}
  \bow{\mu}{\lambda} = \{(b, H)\,:\, H \in \subg, b \in B_{H}\},
\end{align}
where the measure is not displayed explicitly.

\begin{proof}[Proof of Proposition~\ref{prop:bowen-ift}]
  By Proposition~\ref{prop:bowen-space-if}, $\bow{\mu}{\lambda}$ is
  indeed an intermediate factor. Hence there exists a $G$-isomorphism
  $\varphi \colon \bow{\mu}{\lambda} \to (Z,\xi) \times
  (\subg,\lambda)$, where $(Z,\xi)$ is a
  $(G,\mu)$-boundary, and such that the following diagram commutes.
  \begin{align*}
    \xymatrix{ \pb{G,\mu} \times (\subg,\lambda) \ar[d]^{\hat\kappa}
      \ar[dr]^{\kappa_0 \times \id}& \\
      \bow{\mu}{\lambda} \ar[d]^{\hat\pi_r}
      \ar[r]^{\varphi}& (Z,\xi)\times(\subg,\lambda) \ar[dl]^{\pr} \\
      (\subg,\lambda)&}
  \end{align*}
  It follows that $(B_H,\nu_H)$ (the $\hat\pi_r$-fiber above $H$) is
  isomorphic to $(Z,\xi)$ (the $\pr$-fiber above $H$), for
  $\lambda$-almost every $H \in \subg$.

  Note that for $h \in H$,
  \begin{align}
    \label{eq:h-invariant}
    h\kappa((g_1,g_2,\ldots),H) &= h(H g_1, H g_2, \ldots)\\
    &= (H^h h g_1, H^h h g_2, \ldots) \nonumber \\
    &= \kappa((g_1,g_2,\ldots),H). \nonumber
  \end{align}
  We extend $\kappa$ (as in the proof of
  Proposition~\ref{prop:bowen-space-if}) but this time to a factor
  $\tilde \kappa$ from $G^\N \times \subg$ to
  $\bow{\mu}{\lambda}$. Then $\tilde\kappa((g_1,g_2,\ldots),H) = (b,H)$ for
  some $b \in B_{H}$. This follows from the fact that
  $\hat\pi_r(b,H) = H$, and that $\pi_r \circ \kappa$ is the
  projection. It then further follows from~\eqref{eq:h-invariant} that
  $h(b,H) = (b,H)$, for all $h \in H$; the invariance of
  $(H g_1, H g_2,\ldots)$ to $H$ extends from
  $(\cosg^\N,\mathbb{P}^\lambda_\mu)$ to its shift-invariant
  sub-sigma-algebra and its Mackey realization.

  Since, as we noted above, $(B_H,\nu_H)$ is isomorphic to $(Z,\xi)$,
  it follows that $h(z,H) = (z,H)$ for all $\lambda$-almost every
  $h \in H$ and $\xi$-almost every $z \in Z$. But the $G$-action on
  $(Z,\xi)\times(\subg,\lambda)$ is the diagonal action, and so we
  have that $(Z,\xi)$ is $H$-invariant for $\lambda$-almost every
  $H \in \subg$.

  By a standard argument~\cite{varadarajan1963groups}, we can choose a
  compact model for $(Z,\xi)$ such that $G \curvearrowright (Z,\xi)$
  is continuous, and deduce that $(Z,\xi)$ is $H$-invariant for every
  $H$ in the support of $\lambda$, and is therefore
  $\nc{\lambda}$-invariant.
  
\end{proof}

\subsection{IFT actions and co-amenable IRSs}

Given the connection between intermediate factors and Poisson bundles,
and given Theorem~\ref{thm:co-amenable-irss}, we are now ready to show
how intermediate factor theorems can be used to show that an IRS is
co-amenable.
\begin{theorem}
  \label{thm:co-amenable-ift}
  Let $G$ be an lcsc group. Let $K \leq G$
  be an IRS with distribution $\lambda$ such that $G\curvearrowright
  (\subg,\lambda)$ is an IFT action. Then $K$ is co-amenable in
  $\nc{\lambda}$, the normal closure of $\lambda$.
\end{theorem}
\begin{proof}
  By Proposition~\ref{prop:bowen-ift}, if $G\curvearrowright
  (\subg,\lambda)$ is an IFT action then the action of $\nc{\lambda}$
  on the Poisson bundle $\bow{\mu}{\lambda}$ is measure
  preserving. Applying Theorem~\ref{thm:co-amenable-irss} yields the
  desired result.
\end{proof}

\section{Proofs of main theorems and corollaries}
\label{sec:proof-of-corollaries}



\begin{proof}[Proof of Theorem~\ref{thm:product-irs-rigidity}]
  Let $\lambda \in \irsg$ be the distribution of $K$, and denote $N =
  \nc{\lambda}$.  By the Bader-Shalom IFT (Theorem~\ref{thm:ift}),
  Theorem~\ref{thm:co-amenable-ift} implies that $K$ is co-amenable in
  $N$.

  Denote by $N_i \lhd G_i$ the closure of the projection of
  $N$ on $G_i$. By Fact~\ref{fact:normal}, $N$
  is co-abelian in $N_1 \times N_2$, and hence in particular
  co-amenable.  As $K$ is co-amenable in $N$, we conclude
  that $K$ is co-amenable in $N_1\times N_2$.
\end{proof}

\begin{proof}[Proof of Theorem~\ref{thm:lie-irs-rigidity}]
  Let $\lambda \in \irsg$ be the distribution of $K$, and denote
  $N = \nc{\lambda}$.  By the Nevo-Zimmer IFT (Theorem~\ref{thm:ift}),
  Theorem~\ref{thm:co-amenable-ift} implies that $K$ is co-amenable in
  $N$. Hence if $N = G$ then $K$ is co-amenable in $G$, and we are
  done.

  Otherwise $N \neq G$. Then there exists a non-central, closed normal
  $M \lhd G$ such that $N$ and $M$
  commute~\cite{ragozin1972normal}. By the irreducibility of the IRS
  $M$ acts ergodically on $(\subg,\lambda)$. But since $N$ and $M$
  commute, and since $K$ is contained in $N$, this action is
  trivial. Hence $K$ must equal a closed normal subgroup.
\end{proof}

\begin{proof}[Proof of Corollary~\ref{thm:just-non-amenable}]
  Let $K$ be an irreducible IRS of $G$. By
  Theorem~\ref{thm:product-irs-rigidity}, there exist subgroups
  $N_1 \lhd G_1$ and $N_2 \lhd G_2$ such that $K$ is co-amenable in
  $N = N_1 \times N_2$. Since $K$ is irreducible, if either $N_1$ or
  $N_2$ is trivial then $K$ must equal $N_1 \times N_2$. Otherwise,
  because $G_1$ and $G_2$ are just non-amenable, $N_1$ is co-amenable
  in $G_1$ and $N_2$ is co-amenable in $G_2$. Hence $N$ is co-amenable
  in $G$, and so $K$ is co-amenable in $G$.
\end{proof}

To prove Corollary~\ref{thm:discrete} we will need the following two
elementary lemmas. The second one,
Lemma~\ref{lem:G1-essentially-transitive}, is a variation
on~\cite[Lemma 4.2]{bader2006factor} and~\cite[Lemma
1.8]{stuck1994stabilizers}.
\begin{lemma}
  \label{lem:act-transitively}
  Let $H$ be a subgroup of $G_1 \times G_2$. If the projection of $H$
  to the second coordinate is equal to $G_2$, then the action of $G_1$
  on $G/H$ is transitive.
\end{lemma}
\begin{proof}
  The projection of $H$ to the second coordinate is equal to $G_2$;
  that is, for every $g_2 \in G_2$ there exists an $g_1 \in G_1$ such
  that $g_1 g_2 \in H$. Fix any $g_1 g_2 H \in G/H$. Then there exists
  a $g_1' \in G_1$ such that $g_1' g_2^{-1} \in H$, and
  \begin{align*}
    g_1 g_1' H = g_1 g_1' (g_2 g_2^{-1})  H = g_1  g_2  (g_1' g_2^{-1})
    H = g_1 g_2 H,
  \end{align*}
  where the second equality uses the commutativity of $G_1$ and $G_2$,
  and the last equality follows from the fact that
  $g_1' g_2^{-1} \in H$.  Hence $G_1$ indeed acts transitively on
  $G / H$.
\end{proof}
\begin{lemma}
  \label{lem:G1-essentially-transitive}
  Let $G = G_1 \times G_2$ be lcsc with $G_1$ simple, and let
  $G \curvearrowright (X,\im)$ be a non-trivial, irreducible pmp
  action. Then the action $G_1 \curvearrowright (X,\im)$ is
  essentially free.
\end{lemma}
\begin{proof}
  Consider the map $s \colon X \to \sub{G_1}$ which assigns to each
  $x \in X$ its $G_1$-stabilizer
  $s(x) = \{g_1 \in G_1\,:\,g_1 x = x\}$.  This map is easily seen to
  be $G_2$-invariant, by the fact that $G_1$ and $G_2$
  commute. Likewise, it is $G_1$-equivariant, and so, by ergodicity,
  $s(x)$ is $\im$-almost surely equal to some closed normal
  $N_1 \lhd G_1$.
  
  Now, if $N_1 = G_1$ then the action $G_1 \curvearrowright (X,\im)$
  is trivial, and by irreducibility $G \curvearrowright (X,\im)$ is
  trivial, in contradiction to the claim hypothesis. Otherwise, since
  $G_1$ is simple, $N_1$ is trivial, and so the action
  $G_1 \curvearrowright (X,\im)$ is essentially free.
\end{proof}

\begin{proof}[Proof of Corollary~\ref{thm:discrete}]
  Let $G = G_1 \times G_2$, let $G \curvearrowright (X,\im)$ be an
  irreducible pmp action, and let $\lambda$ be the distribution of the
  IRS $K$ associated with $(X,\im)$.  By
  Corollary~\ref{thm:just-non-amenable}, $K$ is either equal to a
  normal subgroup, or $K$ is co-amenable in $G$. In the former case,
  since $G_1$ and $G_2$ are simple, $K$ is either trivial, in which
  case the action is essentially free, or else $K$ is one of
  $\{G_1,G_2,G\}$, and hence by irreducibility the action is
  trivial. We therefore assume henceforth that $K \neq G$ is
  co-amenable in $G$, and show that this leads to a contradiction,
  unless the action is trivial or free.

  Let $\pr_2 \colon \subg \to \sub{G_2}$ be the map that assigns to
  each $H \in \subg$ its projection on the second coordinate. Note
  that $G_2$ is discrete, and therefore $\pr_2 H$ is closed, since
  every subgroup of $G_2$ is closed. This map is easily seen to be
  $G_1$-invariant and $G_2$-equivariant, and therefore $\pr_2 K$ is
  almost surely equal to some normal $M_2 \lhd G_2$, by
  the ergodicity of the $G_2$-action\footnote{We would like to thank
    Jesse Peterson for suggesting this argument.}.

  Assume first that $M_2$ is the trivial subgroup. Then $K$ is
  contained in $G_1$, and so, by the ergodicity of the $G_2$-action on
  $(\subg,\lambda)$, $K$ must almost surely equal some closed normal
  $M_1 \lhd G_1$. As above, if $M_1=G_1$ then the action is
  trivial. Otherwise $M_1$ is trivial, and we have that $K$ is
  trivial, so that the action $G \curvearrowright (X,\im)$ is
  essentially free.
  
  Finally, consider the case that $M_2$ is not the trivial
  subgroup. Since $G_2$ is simple, $M_2 = G_2$, and the projection of
  $K$ on the second coordinate is $G_2$. Therefore, by
  Lemma~\ref{lem:act-transitively}, $G_1$ acts transitively on $G/K$.
  By Lemma~\ref{lem:G1-essentially-transitive}, $G_1$ acts essentially
  freely on $G/K$. Hence the action $G_1 \curvearrowright G/K$ is
  isomorphic to the natural action $G_1 \curvearrowright G_1$. Since
  $K$ is almost surely co-amenable in $G$, there is a $G$-invariant
  mean on $G/K$, which in particular is $G_1$-invariant. Therefore,
  and since we identified $G/K$ with $G_1$, there exists a
  $G_1$-invariant mean on $G_1$. So $G_1$ is amenable - contradiction.
\end{proof}

In order to prove Corollaries~\ref{thm:rigidity-special}
and~\ref{thm:stuck-zimmer} we will use (and elaborate on)
Varadarajan's ergodic decomposition
theorem~\cite{varadarajan1963groups} to show that the existence of
invariant probability measures supported on orbits implies that an
action is essentially transitive. We state Varadarajan's theorem after
the following definition.
\begin{definition}[Varadarajan~\cite{varadarajan1963groups}]
  Let $G$ be an lcsc group acting on a standard measurable space $X$.
  Let $E(X)^G$ denote the space of $G$-invariant, ergodic probability
  measures on $X$. A {\em decomposition map} is a measurable map
  $\beta : X \to E(X)^G$ with the following properties.
  \begin{enumerate}
  \item $\beta$ is $G$-invariant. I.e., $\beta_{g x} = \beta_x$ for all
    $g \in G$ and $x \in X$.
  \item For every $\eta \in E(X)^G$, it holds that
    $\eta(\beta^{-1}(\eta))=1$.
  \item For every $G$-invariant measure $\theta$ it holds that
    \begin{align}
      \label{eq:decompo}
      \theta = \int_X\beta_x\dd\theta(x).
    \end{align}
  \end{enumerate}
\end{definition}
\begin{theorem}[Varadarajan~\cite{varadarajan1963groups}]
  \label{thm:varadarajan}
  For every action of an lcsc group $G$ on a standard measurable space
  $X$, there exists a decomposition map $\beta$. Furthermore, $\beta$
  is essentially unique, in the sense that if $\beta$ and $\beta'$ are
  decomposition maps then $\theta (\{x \in X\,:\, \beta_x \neq
  \beta'_x\})=0$ for any $G$-invariant probability measure $\theta$.
\end{theorem}

\begin{lemma}
  \label{lem:co-finite-stabs}
  Let $G$ be an lcsc group, and let $G \curvearrowright (X,\im)$ be an
  ergodic pmp action. Assume that there exists a $G$-invariant
  probability measure on the orbit $G x$ for $\im$-almost every
  $x \in X$.  Then the action $G \curvearrowright (X,\im)$ is
  essentially transitive.
\end{lemma}
\begin{proof}
  Let $\beta : X \to E(X)^G$ be a decomposition map of $X$ with
  respect to the $G$-action.  Let $x \in X$ be such that there exists
  a $G$-invariant and ergodic probability measure $\eta_x$ with
  $\eta_x(G x)=1$.
  
  By the second property of decomposition maps, there exists an
  element $y \in G x$ (in fact a $\eta_x$-full measure set of such
  elements) for which $\beta_y = \eta_x$. Since $\beta$ is
  $G$-invariant, we get that $\beta_x=\eta_x$ and in particular
  $\beta_x$ is supported on $G x$.

  Let $A$ be an $\im$-full measure set of $x \in X$ for which there
  exists a $G$-invariant, ergodic measure $\eta_x$ on $G x$, and for
  which, by the above, $\beta_x=\eta_x$ is supported on a
  $G$-orbit. Then
  \begin{align*}
    \im = \int_X\beta_x\dd\im(x) = \int_A\beta_x\dd\im(x),
  \end{align*}
  and it follows by the ergodicity of $\im$ that it is equal to some
  $\beta_x$. Hence $\im$ is supported on a $G$-orbit, or, equivalently,
  the action is essentially transitive.
\end{proof}

The following is a corollary of
Lemma~\ref{lem:co-finite-stabs}. Recall that $H$ is co-finite in $G$
if there exists a finite, $G$-invariant measure on $G/H$.
\begin{corollary}
  \label{cor:tame}
  Let $G$ be an lcsc group, and let an ergodic IRS $K$ in $G$ be
  co-finite. Then its distribution is supported on a single orbit
  $\{\fk^g\,:\, g\in G\}$.
\end{corollary}
In Stuck-Zimmer this is proven for the case of Lie
groups~\cite[Corollary 3.2]{stuck1994stabilizers}.

We next prove an analogue of Lemma~\ref{lem:co-finite-stabs} for
almost direct products.
\begin{definition}
  Let $G_1$ and $G_2$ be closed subgroups of $G$. Then $G$ is said to
  be an {\em almost direct product} of $G_1$ and $G_2$ if the two
  groups commute and if $G=G_1G_2$.
\end{definition}

\begin{lemma}
  \label{lem:co-finite-stabs2}
  Let $G=G_1G_2$ be an lcsc almost direct product, and let $G
  \curvearrowright (X,\im)$ be a pmp action on a standard measurable
  space that is $G_1$-ergodic. Assume that there exists a
  $G_1$-invariant probability measure on $\im$-almost every
  $G$-orbit. Then the action $G \curvearrowright (X,\im)$ is
  essentially transitive.
\end{lemma}
\begin{proof}
  Let $\beta^1 : X \to E(X)^{G_1}$ be a decomposition map of the
  $G_1$-action on $X$.  Note that, since $G_1$ and $G_2$ commute, it
  holds for every $\eta \in E(X)^{G_1}$ and $g \in G$ that $g\eta \in
  E(X)^{G_1}$, and so $G$ acts on $E(X)^{G_1}$. Hence one can consider
  the question of whether $\beta^1$ commutes with $G$.
  
  In fact, in order to follow the same arguments of
  Lemma~\ref{lem:co-finite-stabs} we will require that $\beta^1$ be
  essentially $G$-equivariant. That is, that there exists an $A
  \subseteq X$ with $\theta(A)=1$ for any $G_1$-invariant measure
  $\theta$, and such that $g\beta^1_x =\beta^1_{g x}$ for all $g\in G$
  and $x\in A$.  Assume first that $\beta^1$ satisfies this condition.

  Since $m$ is $G_1$-invariant, $\im(A)=1$. Hence there exists an
  $\im$-full measure set $A' \subseteq A$ such that there exists a
  $G_1$-invariant and ergodic measure $\eta_x$ on $G x$ for all $x \in
  A'$.  Fix some $x \in A'$.  Since $\eta_x(G x)=1$, we can find some
  $y = g x \in G x$ such that $\beta^1_y=\eta_x$.  Then
  \begin{align*} 
    g\beta^1_x = \beta^1_{g x} = \beta^1_y = \eta_x
  \end{align*}
  and so we conclude that $\beta^1_x$ is supported on $G x$ for every
  $x \in A'$.

  Since $\im$ is $G_1$-invariant we can write
  \begin{align*}
    \im = \int_X\beta^1_x\dd\im(x) = \int_{A'}\beta^1_x\dd\im(x)
  \end{align*}
  and, by the $G_1$-ergodicity of $\im$, $\im=\beta^1_x$ for some $x
  \in A'$ and in particular $\im(G x)=1$. Hence the action $G
  \curvearrowright (X,\im)$ is essentially transitive.

  Finally, we argue that $\beta^1$ is essentially $G$-equivariant.
  Following Varadarajan, let $\mathcal{F}$ be the Banach space of all
  bounded measurable functions on $X$.  Let $U^1:\mathcal{F}\to
  \mathcal{F}$ be the operator defined by $U^1(f)(x)= \int_X f(y)
  \dd\beta^1_x(y)$, for any $f\in \mathcal{F}$. Note that by definition,
  the equivariance of $\beta^1$ is equivalent to the equivariance of
  $U^1$.

  Now, for any $G_1$-invariant measure $\theta$, $U^1$ is the
  conditional expectation defined by the factor $(X,\theta)\to
  \ergcom{(X,\theta)}{G_1}$, the space of $G_1$-ergodic components of
  $(X,\theta)$:
  \begin{align*}
  	U^1:L^\infty(X,\theta) \to L^\infty(\ergcom{(X,\theta)}{G_1}).
  \end{align*}
  Since the actions of $G_1$ and $G_2$ on $X$ commute, $U^1$, as the
  conditional expectation map, is $G_2$-equivariant. Hence $U^1 :
  \mathcal{F} \to \mathcal{F}$ is $G$-equivariant on a $\theta$-full
  measure set. Since this holds for any $G_1$-invariant measure
  $\theta$, we conclude that $U^1$, and so the associated
  decomposition map $\beta^1$, are essentially $G$-equivariant, and
  the proof is complete.
  
\end{proof}

As a final lemma before the proofs of
Corollaries~\ref{thm:rigidity-special} and~\ref{thm:stuck-zimmer} we
note the following fact regarding property (T) groups. It is a
straightforward generalization of the fact that any property (T) group
that is also amenable must admit a finite invariant measure.
\begin{claim}
  \label{clm:property-t}
  Let $G$ be an lcsc group, let $H \in \subg$ be co-amenable. Let
  $G' \in \subg$ be any closed subgroup with property (T). Then there
  exists a $G'$-invariant probability measure on $G/H$. In particular,
  if $G$ has property (T) then $H$ is co-finite in $G$.
\end{claim}
\begin{proof}
  Since $H$ is co-amenable in $G$, the quasi-regular representation of
  $G$ on $L^2(G/H)$ weakly contains the trivial representation;
  equivalently, there are almost-invariant non-zero vectors in this
  representation (see~\cite[Theorem 4.18 and Definition
  5.1]{shalom1999invariant}).

  Restrict this representation to a representation of $G'$. Then this
  restricted representation also has almost-invariant non-zero
  vectors. Since $G'$ has property (T) this implies that there also
  exist $G'$-invariant unit vectors.

  A unit $G'$-invariant vector in this representation corresponds to a
  function $f \in L^2(G/H,\nu)$, for some $G$-quasi-invariant measure
  $\nu$, such that for all $g \in G'$ and $\nu$-almost every
  $x \in G/H$
  \begin{align*}
    f(x) = f(g^{-1}x)\sqrt{\frac{\dd g_*\nu}{\dd\nu}(x)},
  \end{align*}
  and such that $\int |f|^2 d\nu  = 1$.
  
  Define $\nu'$ by $\dd\nu'(x) = |f(x)|^2 \dd\nu(x)$. Then for any $g
  \in G'$
  \begin{align*}
    \dd(g_*\nu')
    &= |f(g^{-1}x)|^2\dd\nu(g^{-1}x)\\
    &= |f(x)|^2\frac{\dd \nu}{\dd g_*\nu}(x)\dd(g_*\nu)(x)\\
    &= |f(x)|^2\dd\nu(x)\\
    &= \dd\nu'.
  \end{align*}
  Hence $\nu'$ is $G'$-invariant.  Finally,
  \begin{align*}
    \nu'(G/H) = \int_{G/H}|f(x)|^2\dd\nu(x) = 1, 
  \end{align*}
  and so $\nu'$ is a probability measure.
\end{proof}

We are now ready to prove Corollaries~\ref{thm:rigidity-special}
and~\ref{thm:stuck-zimmer}.
\begin{proof}[Proof of Corollary~\ref{thm:rigidity-special}]
  Let $K$ be the associated stabilizer IRS. By
  Theorem~\ref{thm:product-irs-rigidity}, there exists a normal
  subgroup $N = N_1 \times N_2 \lhd G$ such that $K$ is co-amenable in
  $N$.

  Since $G_1$ is just non-amenable, either $N_1$ is co-amenable in
  $G_1$ or else $N_1$ is trivial. In the latter case we have that $K$
  is contained (in fact, co-amenable) in $\{e\} \times N_2$, and 
  since $K$ is irreducible, it must almost surely equal $\{e\} \times
  N_2'$, for some $N_2' \lhd G_2$. Since the action $G
  \curvearrowright (X,\im)$ is faithful, $N_2' = \{e\}$ and the action
  is essentially free.
  
  We are therefore left with the case that $N_1$ is co-amenable in
  $G_1$.  Then since $K$ is co-amenable in $N_1 \times N_2$ we get
  that $K$ is co-amenable in $G_1\times N_2$. It follows that, since
  $G_1$ has property (T), $m$-almost every $G_1 \times N_2$-orbit
  admits a $G_1$-invariant probability measure
  (Claim~\ref{clm:property-t}), and so each $G$-orbit admits a
  $G_1$-invariant probability measure. It follow from
  Lemma~\ref{lem:co-finite-stabs2} that $G \curvearrowright (X,\im)$
  is essentially transitive, and hence $K$ is co-finite.
\end{proof}

\begin{proof}[Proof of Corollary~\ref{thm:stuck-zimmer}]
  Assume that $G_1$ is a simple factor of $G$ with property (T), and
  note that by Theorem~\ref{thm:lie-irs-rigidity} $K$ is either almost
  surely equal to a normal subgroup $N$, or $K$ is co-amenable in $G$. 

  In the former case, since the action is faithful, the only
  possibility is $N=\{e\}$, and then the action is essentially free.

  In the latter case, the same argument of the proof of
  Corollary~\ref{thm:rigidity-special} above shows that
  $G \curvearrowright (X,\im)$ is essentially transitive and $K$ is
  co-finite in $G$.

  Finally, it follows from the Borel Density Theorem that $K$ is
  discrete, and hence a lattice in $G$.
\end{proof}

\bibliography{irs_rigidity}
\end{document}